\documentclass[11pt]{article}
\usepackage{graphicx}
\usepackage{amsfonts,amsmath,amssymb,amscd,enumitem,color}
\usepackage{mathtools}
\usepackage{accents}
\usepackage{amsthm}
\usepackage[hidelinks]{hyperref}
\usepackage[capitalize]{cleveref}
\usepackage[english]{babel}
\usepackage[utf8]{inputenc}
 
\newtheorem{theorem}{Theorem}
\newtheorem{proposition}[theorem]{Proposition}
\newtheorem{definition}[theorem]{Definition}

\newtheorem{corollary}[theorem]{Corollary}

\numberwithin{subcase}{case}

\numberwithin{subsubcase}{subcase}

\usepackage{fullpage}

\author{Brett Leroux}

\def\keywords#1{\par\addvspace\medskipamount{\rightskip=0pt plus1cm
\def\and{\ifhmode\unskip\nobreak\fi\ $\cdot$
}\noindent{Keywords:}\enspace\ignorespaces#1\par}}

\title{The minimum neighborliness of a random polytope}

\usepackage{dirtytalk}

\def\final{0}  

\ifnum\final=0  
\newcommand{\lnote}[1]{[{\small Luis: \bf #1}]}
\newcommand{\bnote}[1]{[{\small Brett: \bf #1}]}
\newcommand{\anonnote}[1]{[{\small anon: \bf #1}]}
\newcommand{\sidecomment}[1]{\marginpar{\tiny #1}}
\newcommand{\details}[1]{{\color{blue}\ [[#1]] }}
\else 
\newcommand{\lnote}[1]{}
\newcommand{\bnote}[1]{}
\newcommand{\anonnote}[1]{}
\newcommand{\sidecomment}[1]{}
\newcommand{\details}[1]{}
\fi  

\newcommand{\Rl}{\operatorname{\mathbb{R}}}

\newcommand{\Tor}{\operatorname{\Tor}}

\newcommand{\aff}{\operatorname{aff}}
\newcommand{\conv}{\operatorname{conv}}

\newcommand{\interior}{\operatorname{int}}
\newcommand{\suchthat}{\mathrel{:}}

\newcommand{\ud}{\mathop{}\!\mathrm{d}}
\newcommand{\prob}{\mathbb{P}}

\makeatletter
\newcommand\fake@math{}
\def\fake@math#1\){[math]}
\makeatother
\usepackage{color}

\begin{document}
\maketitle
\begin{abstract}
Let $\mu$ be a probability distribution on $\Rl^d$ which assigns measure zero to every hyperplane and $S$ a set of points sampled independently from $\mu$. What can be said about the expected combinatorial structure of the convex hull of $S$? These polytopes are simplicial with probability one, but not much else is known except when more restrictive assumptions are imposed on $\mu$. In this paper we show that, with probability close to one, the convex hull of $S$ has a high degree of neighborliness no matter the underlying distribution $\mu$ as long as $n$ is not much bigger than $d$. As a concrete example, our result implies that if for each $d$ in $\mathbb{N}$ we choose a probability distribution $\mu_d$ on $\Rl^d$ which assigns measure zero to every hyperplane and then set $P_n$ to be the convex hull of an i.i.d. sample of $n \le 5d/4$ random points from $\mu_d$, the probability that $P_n$ is $k$-neighborly approaches one as $d \to \infty$ for all $k\le d/20$.
We also give a simple example of a family of distributions which essentially attain our lower bound on the $k$-neighborliness of a random polytope. 
\end{abstract}

\section{Introduction}

The most well-studied models of random polytopes are those where a random polytope is defined as the convex hull of a set of independent and identically distributed points from some probability distribution on the space $\Rl^d$. These objects have been studied for several reasons. One reason is that random polytopes can sometimes give us some insight into the possible metric or combinatorial properties of deterministic convex polytopes with a given dimension and number of vertices. Another is that convex polytopes have a wide range of applications in geometric algorithms including the simplex method \cite{MR0868467} and Wolfe’s method \cite{MR4526301}. For many such algorithms, the input data defines a convex polytope and it is useful to understand combinatorial and metric properties of that polytope in order to understand the complexity of the algorithm. Since algorithmic applications often assume that the input data is random according to some predetermined distribution, random polytopes are particularly relevant.

Gaussian random polytopes, i.e. random polytopes where the underlying distribution is the standard Gaussian distribution on $\Rl^d$ have received much attention \cite{MR1149653,MR2330981,MR2093024,MR2144555,MR0258089,MR0169139}. Perhaps the main reason for this focus on Gaussian random polytopes is that they coincide in distribution with uniform random projections of a simplex to some lower dimensional space \cite{MR1254086}. Another reason is that the Gaussian distribution has many nice properties which often make calculations simpler and behavior of combinatorial properties easier to determine. Other commonly studied families of random polytopes are those for which the underlying distribution is the uniform distribution on a convex body or the boundary of a convex body. Yet another important example is \emph{random $0/1$ polytopes} \cite{MR2144601} where the vertices are a random subset of the vertices of the $d$-dimensional cube. Some papers prove results which assume only that the underlying distribution has some property such as being log-concave \cite{BBB} or subgaussian \cite{MR2373017}.

In all of these examples, either the distribution is specified, or it satisfies some restrictive condition such as being subgaussian. Our main result in contrast assumes only that the distribution assigns measure zero to every (affine) hyperplane. With this assumption only, we consider random polytopes where the number of random points is proportional to the dimension and the dimension approaches infinity. The main property of convex polytopes we are interested in is $k$-neighborliness (defined in \cref{sec:Gauss}). One of our main results shows that if the constant of proportionality is less than two, then there exists some constant $\beta>0$ (depending on the constant of proportionality), such that the probability that the random polytope is at least $\lfloor \beta d\rfloor$-neighborly approaches one as the dimension approaches infinity (\cref{thm:neighborly}). 

To put this result in context, we need to review what has previously been known about the neighborliness of random polytopes. First we collect some notation.

\subsection{Notation}
Asymptotic notation $f(n) \sim g(n)$ means $f(n)/g(n) \to 1$ as $n \to \infty$. For a set of points $X \subset \Rl^d$, $\conv X$ is the convex hull of $X$. Similarly, $\aff X$ is the affine hull of $X$. The \emph{binary entropy function} is the function defined by $H(p) := -p\log_2 p -(1-p)\log_2(1-p)$. We use $o$ to denote the origin in $\Rl^d$. For a polytope $P$, we use the notation $f_\ell(P)$ for the number of $\ell$-dimensional faces of the polytope $P$. When we say that a probability distribution assigns measure zero to every hyperplane we mean every affine hyperplane (not just every linear hyperplane) unless otherwise specified.

\subsection{Previous work on the neighborliness of random polytopes}\label{sec:Gauss}

\begin{definition}
A polytope $P$ is \emph{$k$-neighborly} if every subset of at most $k$ vertices is a face of the polytope. 
\end{definition}

See \cite{MR1976856} for an introduction to polytopes and \cite[Chapter 7]{MR1976856} for $k$-neighborly polytopes in particular. 

In addition to our results about $k$-neighborliness, we will also consider another quantity associated to polytopes which measures how \say{close} the polytope is to being $k$-neighborly. 

\begin{definition}
For a simplicial polytope $P \subset \Rl^d$ with $n$ vertices and any $0\le \ell \le d-1$, the \emph{$\ell$-face density} of $P$ is $f_\ell(P)/\binom{n}{\ell+1}$.
\end{definition}

We see that the $\ell$-face density of a polytope measures how close to being $k$-neighborly the polytope is where $k = \ell+1$. If the $\ell$-face density is one, then the polytope is $(\ell+1)$-neighborly. In addition to showing that a random polytope is $k$-neighborly for a surprisingly large value of $k$, we will show that a random polytope has $\ell$-face density close to one where $\ell$ is even larger than $k$. 

It was perhaps Gale who first speculated that random polytopes should have a high degree of neighborliness when the dimension of the space is high \cite{MR0085552}. 
More recently, it has been rigorously proven that random polytopes from certain families of probability distributions tend to have a surprisingly high degree of neighborliness with high probability \cite{MR2796091,MR0927239,MR2546388,MR2168716,MR2449053,MR1976856,MR0085552,MR2373017,MR2206919,VS}. Two of these works, the paper \cite{MR2168716} of Donoho and Tanner, and the paper \cite{VS} of Vershik and Sporyshev are particularly relevant to this paper so we give an overview of their results.

A \emph{Gaussian random polytope} is the convex hull of an i.i.d. sample of points from the standard (mean zero and identity covariance matrix) Gaussian distribution on $\Rl^d$. Donoho and Tanner show in \cite{MR2168716} that there exists a function $\rho_{DT}(\delta)$ such that if $\rho<\rho_{DT}(\delta)$ and $G_{n,d}$ is a Gaussian random 
polytope in $\Rl^d$ with $n$ random points with $d \ge \delta n$, then the probability that $G_{n,d}$ is $\lfloor \rho d\rfloor$-neighborly approaches one as $d$ approaches infinity 
(\cite[Theorem 1]{MR2168716}). Furthermore, they show that if $\rho>\rho_{DT}(\delta)$, then the expected number of subsets of the points of size $\lfloor \rho d\rfloor$ which are not faces of the polytope approaches infinity as $d \to \infty$.

Vershik and Sporyshev establish a similar result for the $\ell$-face density. They show that there exists a function $\rho_{VS}(\delta)$ such that if $d =d(n)\sim \delta n$ and $\ell =\ell(n) \sim \rho d$ where $\rho<\rho_{VS}(\delta)$, then the expected $\ell$-face density of $G_{n,d}$ approaches one as $d$ approaches infinity and that if $\rho>\rho_{VS}(\delta)$, then the expected $\ell$-face density approaches zero as $d$ approaches infinity \cite[Theorem 1]{VS}.

\subsection{Our results}
Our two main results are similar to the two results explained in the previous section, the first due to Donoho and Tanner, and the second due to Vershik and Sporyshev. The main difference is that our results apply to any random polytope whose vertices are i.i.d. according to a probability distribution on $\Rl^d$ which assigns measure zero to every hyperplane. This is a very weak assumption on the distribution. In particular, it is the minimal assumption which guarantees that the random polytopes under consideration are simplicial with probability one. Not surprisingly, because of the generality of the distributions we consider, our result guarantees a lower degree of neighborliness than in the case where the distribution is Gaussian. 
\begin{theorem}\label{thm:neighborly}
Let $\alpha <2$ and assume that $\beta>0$ satisfies $\alpha H(\beta/\alpha)+(\alpha-\beta)(H(\frac{\alpha-1}{\alpha-\beta})-1)<0$, or equivalently, $\frac{\alpha^\alpha}{(\alpha-1)^{\alpha-1}2^\alpha} < \frac{\beta^\beta(1-\beta)^{1-\beta}}{2^\beta} $. For each $d\in \mathbb{N} $, let $\mu_d$ be a probability distribution on $\Rl^d$ which assigns measure zero to every hyperplane. Let $n:=n(d) \sim \alpha d$ and let $S_n$ be a set of $n$ independent and identically distributed points from $\mu_d$. Then for any sequence $k:= k(d)$ with $k \sim \beta d$, the probability that  $\conv S_n$ is $k$-neighborly approaches one as $d \to \infty$.
\end{theorem}
\begin{figure}[h]
    \centering
    \includegraphics[width=0.6\textwidth]{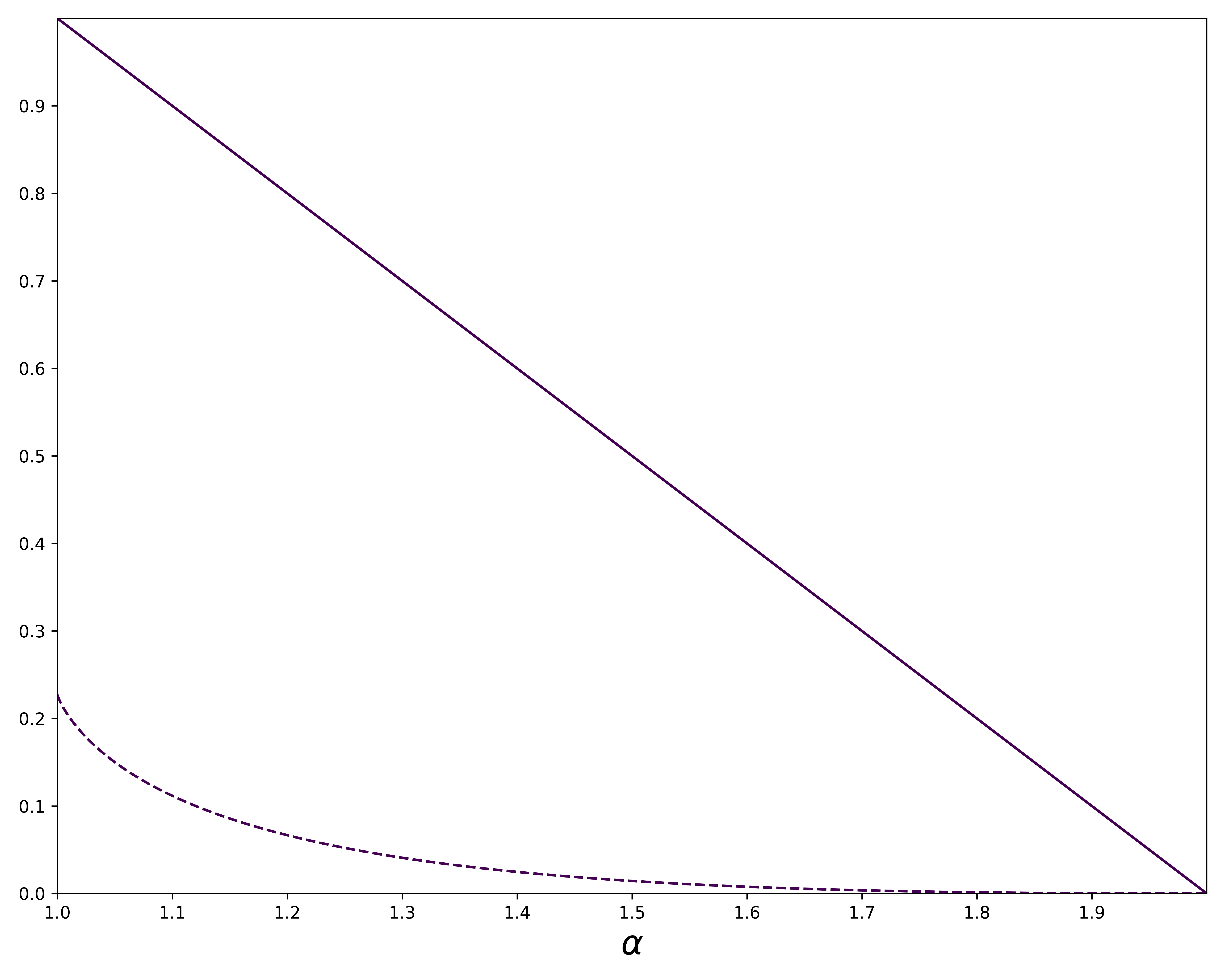}
    \caption{$\rho_N'$ (dashed) and $\rho_D'$. The lower curve shows that value of $\rho_N'(\alpha)$ for each $\alpha \in (1,2)$ which is the lower bound on the neighborliness of a random polytope with $\sim \alpha d$ vertices. The upper curves shows the value of $\rho_D'(\alpha)$ for each $\alpha \in (1,2)$ which is the lower bound on the $\ell$-face density.}
    \label{fig:alpha}
\end{figure}
In \cref{thm:neighborly} the equation $\alpha H(\beta/\alpha)+(\alpha-\beta)(H(\frac{\alpha-1}{\alpha-\beta})-1)=0$ implicitly determines a function which we denote $\rho_N'(\alpha)$ and which is plotted in \cref{fig:alpha}. By \cref{thm:neighborly}, the function $\rho_N'(\alpha)$ has the property that if $\rho< \rho_N'(\alpha)$, then the probability that a random polytope in $\Rl^d$ with $\sim \alpha d$ vertices will be at least $\lfloor \rho d\rfloor$-neighborly approaches one as $d \to \infty$. We have a similar result for the $\ell$-face density: 

\begin{theorem}\label{thm:density}
Let $\alpha<2$ and $0<\beta <2-\alpha$. For each $d\in \mathbb{N} $, let $\mu_d$ be a probability distribution on $\Rl^d$ which assigns measure zero to every hyperplane. Let $n:=n(d) \sim \alpha d$ and let $S_n$ be a set of $n$ independent and identically distributed points from $\mu_d$. Then for any sequence $\ell:= \ell(d)$ with $\ell \sim \beta d$, the expected $\ell$-face density of $\conv S_n$ approaches one as $d \to \infty$.
\end{theorem}
\begin{figure}[h]\label{fig:2}
    \centering
    \includegraphics[width=0.6\textwidth]{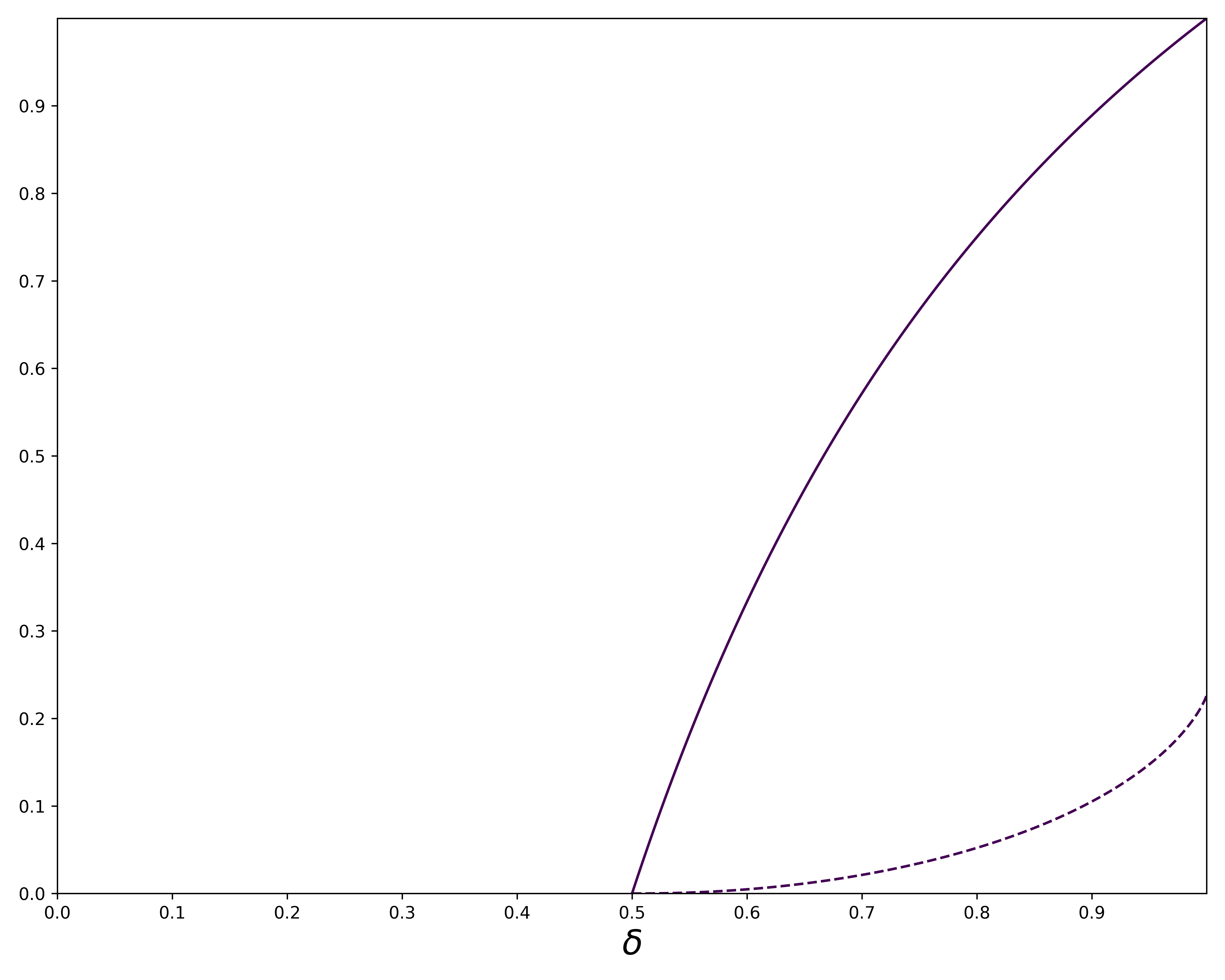}
    \caption{$\rho_N$ (dashed) and $\rho_D$. The lower curve shows that value of $\rho_N(\delta)$ for each $\delta \in (1/2,1)$ which is the lower bound on the neighborliness of a random polytope with $\sim (1/\delta) d$ vertices. The upper curves shows the value of $\rho_D(\delta)$ for each $\delta \in (1/2,1)$ which is the lower bound on the $\ell$-face density.}    
    \label{fig:delta}
\end{figure}
We define the function $\rho_D'(\alpha)= 2-\alpha$. By \cref{thm:neighborly}, the function $\rho_D'(\alpha)$ has the property that if $\rho< \rho_D'(\alpha)$, then the expected $\lfloor \rho d\rfloor$-face density of a random polytope in $\Rl^d$ with $\sim \alpha d$ vertices approaches one as $d \to \infty$. The function $\rho_D'(\alpha)$ is also plotted in \cref{fig:alpha}.

Recall that the functions $\rho_{VS}(\delta)$ and $\rho_{DT}(\delta)$ discussed in \cref{sec:Gauss} are defined as functions of $\delta$ where the dimension $d$ satisfies $d \sim \delta n$ where $n$ is the number of vertices. In contrast, we defined the functions $\rho_N'(\alpha)$ and $\rho_{D}'(\alpha)$ as functions of $\alpha$ where the number of vertices $n$ satisfies $n \sim \alpha d$. Therefore in order to compare the above the results to the results for the Gaussian case discussed in \cref{sec:Gauss}, we set $\delta = 1/\alpha$ and define functions $\rho_{N}(\delta):= \rho_N'(1/\delta)$ and $\rho_{D}(\delta):= \rho_D'(1/\delta)$. These functions are plotted in \cref{fig:delta} and can be compared with the functions $\rho_{VS}(\delta)$ and $\rho_{DT}(\delta)$ which are plotted in \cite[Figure 1]{MR2168716}.

In \cref{sec:bad} we show that the above two results are close to best possible by constructing a family of distributions on $\Rl^d$ which show that \cref{thm:neighborly} (resp. \cref{thm:density}) is not true if $\rho_{N}'(\alpha)$ (resp. $\rho_{D}'(\alpha)$) is replaced by a function which is strictly less than $\rho_{N}'(\alpha)$ (resp. $\rho_{D}'(\alpha)$) for any $\alpha$ in the range $(1,2)$. 

\subsection{Applications}
As discussed earlier, one of the main reasons to study random polytopes is to help understand the average behavior of algorithms where the input data can be thought of as a convex polytope. For the specific property of $k$-neighborliness, the main application is to \emph{compressed sensing}, see \cite{MR2168716} for an explanation of the connection of $k$-neighborliness to compressed sensing. This connection is what has motivated much of the work on the neighborliness of random polytopes some of which was cited in \cref{sec:Gauss}.

\subsection{Outline of the paper}
\cref{sec:lower} contains the proofs of \cref{thm:neighborly,thm:density}. Before proving the theorems, we explain the idea behind the proof in \cref{sec:WW} and also in that section we explain the main tool behind the proofs. The main tool is a  generalization of a result of Wagner and Welzl \cite{WW} which gives an upper bound on the probability that a given point is in the convex hull of a sample of random points.

\cref{sec:least} contains the construction of the distributions which show that \cref{thm:neighborly,thm:density} are essentially best possible. In other words, we construct distributions which produce the least neighborly polytopes over all distributions which assign measure zero to every hyperplane. In order to prove that the distributions we construct have the desired property, we require a result which gives a lower bound on the probability that a given point is in the convex hull of a sample of random points. More specifically, given a distribution on $\Rl^d$ and a point in $\Rl^d$ which is at some \emph{depth} (\cref{def:depth}) with respect to the distribution, we prove a lower bound on the probability that point is in the convex hull of a sample of random points where this probability depends on the depth of the point (\cref{prop:lower}).

\section{The lower bound on the neighborliness}\label{sec:lower}
\cref{sec:WW} explains the main idea behind the proof of \cref{thm:neighborly,thm:density} and outlines the main tool for the proof. The proof is completed in \cref{sec:proofs} after establishing some necessary lemmas. 

\subsection{An upper bound on the probability that a point is in the convex hull}\label{sec:WW}

Wendel's theorem (\cite{MR0146858},\cite[Theorem 8.2.1]{schneider}) is a classic result in geometric probability which says that the probability that a set of $n$ i.i.d. random points from a distribution on $\Rl^d$ which is symmetric about $o$ and which assigns measure zero to every linear hyperplane is equal to $\frac{\sum_{i=0}^{n-d-1}\binom{n-1}{i}}{2^{n-1}}$. More recently, Wendel's theorem has been generalized by Wagner and Welzl in the following way. An \emph{absolutely continuous} probability distribution on $\Rl^d$ is a probability distribution which has a density function with respect to the Lebesgue measure on $\Rl^d$. A measure $\mu$ is \emph{balanced about a point $p$} if every hyperplane through $p$ divides $\mu$ into two equal halves. Wagner and Welzl showed the following
\begin{theorem}[{\cite[Corollary 3.7]{WW}}]\label{thm:WW}
Let $\mu$ be an absolutely continuous probability measure on $\Rl^d$. Let $S$ be a set of $n$ independent and identically distributed points from $\mu$. Then the probability that $\conv S$ contains the origin is at most
\[
\frac{\sum_{i=0}^{n-d-1}\binom{n-1}{i}}{2^{n-1}}
= 1-\frac{\sum_{i=0}^{d-1}\binom{n-1}{i}}{2^{n-1}}.
\]
and this bound is attained if and only if $\mu $ is balanced about the origin. 
\end{theorem}

Why is this sort of result useful for our purposes? In order to prove \cref{thm:neighborly}, we need to show that certain random polytopes are $k$-neighborly with probability close to one. Therefore, we need an upper bound on the probability that the polytope is not $k$-neighborly. By the union bound, this probability is upper bounded by $\binom{n}{k}$ times the probability that a subset $K$ of the vertices of size $k$ is not a face of the polytope. A standard fact about polytopes is that $\conv K$ is not a face of the polytope if and only if the affine hull of $K$ intersects the convex hull of the remaining vertices of the polytope (\cref{prop:face}). Therefore, letting $V$ denote the set of vertices of the polytope, if we project $V$ to the subspace which is orthogonal to the affine hull of $K$, then the event that $\conv K$ is not a face of the polytope is equivalent to the event that the convex hull of $V \setminus K$ contains the projection of the affine hull of $K$ which is a point in the image space of the projection. \cref{thm:WW} can then be used to given an upper bound for the probability of this event. 

Rather than using \cref{thm:WW}, we will prove another version of this result which applies not just to absolutely continuous distribution but to any distribution which assigns measure zero to every hyperplane. We also give a proof of this result which is very different from the proof of \cref{thm:WW} given in \cite{WW}. However the following result is not entirely new; it was mentioned in \cite{WW} that such a proof works but no details were given.

\begin{theorem}\label{thm:nonatomic}
Let $\mu$ be a probability measure on $\Rl^d$ that assigns measure zero to every hyperplane. Let $S$ be a set of $n$ independent and identically distributed points from $\mu$. Then the probability that $\conv S$ contains the origin is at most
\[
\frac{\sum_{i=0}^{n-d-1}\binom{n-1}{i}}{2^{n-1}}.
\]
\end{theorem}

\begin{proof}
The quantity of interest is
\[
\mathbb{P}(o\in \conv(X_1,\dotsc,X_n)).
\]
Note that since $\mu$ assigns measure zero to ever hyperplane (and, in particular, every hyperplane through the origin), the probability that some $\ell$-flat through the origin contains more than $\ell$ points is zero. This means that the probability that $\conv(X_1,\dotsc,X_n)$ contains the origin on its boundary is zero and so the above probability is equal to 
\[
p:=\mathbb{P}(o\in \interior\conv(X_1,\dotsc,X_n)).
\]
Let $X_1,\dotsc, X_N$ be $N$ i.i.d. points distributed according to $\mu$ where $N > n$ is some integer. Now we consider the \emph{Gale transform} \cite[Section 5.4]{MR1976856} of the set of points $\{X_1,\dotsc, X_N\}$. The Gale transform of $\{X_1,\dotsc,X_N\}$ is a set $\bar{X}_N$ of $N$ points in $\Rl^{N-d-1}$ and by \cite[5.4.1]{MR1976856} there is a one-to-one correspondence between subsets of $\{X_1,\dots,X_N\}$ of size $n$ which contain the origin in their interior and subsets of $\bar{X}_N$ of size $N-n$ which are faces of the polytope $\conv \bar{X}_N$. We note that since no hyperplane through the origin contains more than $d-1$ points of $\{X_1,\dotsc,X_N\}$ with probability one, the polytope $\conv \bar{X}_N$ is simplicial with probability one \cite[Section 5.4]{MR1976856}. 

By the Upper Bound Theorem for convex polytopes \cite{UBT}, and the formula for the number of faces of each dimension of a cyclic polytope, see for example \cite[Theorem 17.3.4]{hand}, the number of subsets of $\bar{X}_N$ of size $N-n$ which are faces of the polytope $\conv \bar{X}_N$ is at most
\begin{align}\label{eq:cyclic}
C(N,n,d):= \frac{N-\delta(n-1)}{n}\sum_{j=0}^{\lfloor \frac{N-d-1}{2}\rfloor} \binom{N-1-j}{n-1}\binom{n}{N-2j+\delta}
\end{align}
where $\delta = N-d-1-2\lfloor(N-d-1)/2 \rfloor$. Therefore, we  know that the expected number of subsets of $\{X_1,\dotsc,X_N\}$ of size $n$ which contain the origin in the interior of their convex hull is at most $C(N,n,d)$, i.e., 
\begin{align*}
p\cdot\binom{N}{n} \le C(N,n,d).
\end{align*}
Therefore, in order to prove the desired bound on $p$, it will suffice to show that 
\begin{align}\label{eq:limit}
\lim_{N\to \infty} \frac{C(N,n,d)}{\binom{N}{n}} = \frac{\sum_{i=0}^{n-d-1}\binom{n-1}{i}}{2^{n-1}}.
\end{align}
To make the calculation simpler, we will restrict our attention to values of $N$ such that $N-d-1$ is even, i.e. the parity of $N$ is the opposite of the parity of $d$. This means that $\delta=0$.
For the terms in \cref{eq:cyclic} to be non-zero, we need $N-2j+\delta \le n$ and so $2j \ge N+\delta-n$. Therefore, \cref{eq:cyclic} equals
\[
\frac{N}{n}\sum_{j=\lceil\frac{N-n}{2} \rceil}^{\frac{N-d-1}{2}} \binom{N-1-j}{n-1}\binom{n}{N-2j}
=\frac{N}{n}\sum_{j=\lceil\frac{N-n}{2} \rceil}^{ \frac{N-d-1}{2}} \binom{N-1-j}{n-1}\bigg( \binom{n-1}{N-2j-1} + \binom{n-1}{N-2j} \bigg)
\]
Letting $m = (N-d-1)/2-j$ and using that $(N-d-1)/2-\lceil\frac{N-n}{2} \rceil = \lfloor \frac{n-d-1}{2}\rfloor $, the above is equal to
\begin{align*}
&\frac{N}{n}\sum_{m=0}^{\lfloor \frac{n-d-1}{2}\rfloor } \binom{N-1-  (N-d-1)/2-m }{n-1}\bigg( \binom{n-1}{d+1+2m-1} + \binom{n-1}{d+1+2m} \bigg)\\
&=\frac{N}{n}\sum_{m=0}^{\lfloor \frac{n-d-1}{2}\rfloor } \binom{N/2+(d+1)/2-1-m}{n-1}\bigg( \binom{n-1}{d+1+2m-1} + \binom{n-1}{d+1+2m} \bigg)\\
&\stackrel{N\to \infty}{\sim}\frac{N^{n}}{n!} \frac{1}{2^{n-1}}\sum_{m=0}^{\lfloor \frac{n-d-1}{2}\rfloor }\bigg( \binom{n-1}{d+1+2m-1} + \binom{n-1}{d+1+2m} \bigg)\\
&= \frac{N^{n}}{n!}\frac{1}{2^{n-1}}\sum_{i=0}^{n-d-1}\binom{n-1}{i}.
\end{align*}
Since $\binom{N}{n}\stackrel{N\to \infty}{\sim}\frac{N^n}{n!}$, \cref{eq:limit} follows. 
\end{proof}

\subsection{Proofs of \cref{thm:neighborly,thm:density}}\label{sec:proofs}
In this section we prove \cref{thm:neighborly,thm:density}.  Before proving the theorems we need to make rigorous the idea explained in \cref{sec:WW}. That is we first show how \cref{thm:nonatomic} can be used to upper bound the probability that a subset of vertices of a random polytope is not a face of the polytope. This bound is given in \cref{thm:face}. First we need a simple corollary of \cref{thm:nonatomic}:
\begin{corollary}\label{cor:flat}
Let $\mu$ be a probability measure on $\Rl^d$ that assigns measure zero to every hyperplane. Let $S:= \{X_1,\dotsc,X_n\}$ be a set of $n$ independent and identically distributed points from $\mu$. Let $L\subset \Rl^d$ be some affine $\ell$-flat. Then the probability that $ L \cap \conv S \neq \emptyset$ is at most
\[
\frac{\sum_{i=0}^{n-(d-\ell)-1}\binom{n-1}{i}}{2^{n-1}}.
\]
\end{corollary}
\begin{proof}
Note that any affine coordinate transformation of $\Rl^d$ preserves the fact that $\mu$ assigns measure zero to every hyperplane. Therefore, after applying an affine coordinate transformation to $\mu$, we may assume that $L$ is the span of the first $\ell$ standard basis vectors in $\Rl^d$. We can identify $\Rl^{d-\ell}$ with $\{0,\dotsc,0\}\times \Rl^{d-\ell}\subset\Rl^d$. Let $\pi_{d-\ell}$ denote the orthogonal projection onto $\Rl^{d-\ell}$. 

The probability that $ L \cap \conv S \neq \emptyset$ is equal to the probability that $\conv(\tilde{X_1}, \dotsc,\tilde{X_n})$ contains the origin where the $\tilde{X_i}$ are independent and distributed according to the projection of $\mu$ by $\pi_{d-\ell}$. The corollary now follows from \cref{thm:nonatomic}.
\end{proof}
The above corollary is useful for our purposes due to the following well-known fact: 
\begin{proposition}[{\cite[Exercise 3.1.1]{MR1976856}}]\label{prop:face}
Let $P$ be a polytope and $A\subset V(P)$. Then $\conv A$ is a face of $P$ if and only if $\aff A \cap \conv(V(P) \setminus A) = \emptyset$. 
\end{proposition}

\begin{theorem}\label{thm:face}
Let $\mu$ be a probability measure on $\Rl^d$ that assigns measure zero to every hyperplane. Let $S_n$ be a set of $n$ independent and identically distributed points from $\mu$. Then for any subset $K$ of $S_n$ of size $k$, the probability that $\conv K$ is not a face of $\conv S_n$ is at most

\[
\frac{\sum_{i=0}^{n-d-2}\binom{n-k-1}{i}}{2^{n-k-1}}. 
\]
\end{theorem}
\begin{proof}
Let $S_n:=\{X_1,\dotsc, X_n\}$. Define
\[ f(x_1,\dotsc, x_n):= \begin{cases} 
      1 & \text{ if } \conv(x_1,\dotsc,x_n) \text{ is not a face of } \conv S_n \\
     0 &  \text{ otherwise}
      \end{cases}
\]
Since all the $X_i$ are independent and identically distributed, the probability that $\conv(X_{i_1},\dotsc,X_{i_k})$ is not a face of $\conv S_n$ is independent of the choice of subscripts. The probability that $\conv(X_{n-k+1},\dotsc, X_n)$ is not a face of $\conv S_n$ is equal to
\begin{equation}\label{eq:int}
\int_{\Rl^d} \dotsc\int_{\Rl^d} f(x_{n-k+1},\dotsc,x_n) \mu(\ud x_1)\dotsc \mu(\ud x_n).
\end{equation}
For any fixed choice of points $x_{n-k+1}, \dotsc, x_n$, the inner integral
\[
\int_{\Rl^d} \dotsc\int_{\Rl^d} f(x_{n-k+1},\dotsc,x_n) \mu(\ud x_1)\dotsc \mu(\ud x_{n-k})
\]
is equal (by \cref{prop:face}) to the probability that $\aff(x_{n-k+1},\dotsc, x_n) \cap \conv(S_n \setminus \{x_{n-k+1},\dotsc,x_n\}) \neq \emptyset$. If $\{x_{n-k+1},\dotsc, x_n\}$ are in general position, that is, they are not contained in any affine $(k-2)$-flat, then by \cref{cor:flat}, this probability is at most
\begin{equation}\label{eq:bound}
\frac{\sum_{i=0}^{n-d-2}\binom{n-k-1}{i}}{2^{n-k-1}}. 
\end{equation}
Since $\mu$ assigns measure zero to every hyperplane, the measure of the set of $\{x_{n-k+1}, \dotsc,x_n\}$ which are contained in some affine $(k-2)$-flat is zero. Therefore, \cref{eq:int} is the integral of a function that is bounded by \cref{eq:bound} except possibly on a set of measure zero, so statement of the theorem follows. 
\end{proof}

We can now prove the main results. 

\begin{proof}[Proof of \cref{thm:neighborly}]
    Let $S_n:=\{X_1,\dotsc, X_n\}$. By \cref{thm:face} and the union bound, the probability that $\conv S_n$ is not $k$-neighborly is at most
\begin{align*}
    \binom{n}{k}\frac{\sum_{i=0}^{n-d-2}\binom{n-k-1}{i}}{2^{n-k-1}}.
\end{align*}
So we just need to prove that this quantity goes to zero as $d\to \infty$. The assumptions on $n,k,d$ imply that for $d$ sufficiently large, $n-d-2<\frac{n-k-1}{2}$. Therefore, by the unimodality of the binomial coefficients and using that $\binom{n}{k} \le 2^{nH(k/n)}$, for $d$ sufficiently large, 
\begin{align*}
    \binom{n}{k}\frac{\sum_{i=0}^{n-d-2}\binom{n-k-1}{i}}{2^{n-k-1}}
    & \le  \frac{\binom{n}{k}\cdot n\cdot \binom{n-k-1}{n-d-2}}{2^{n-k-1}}\\
    & \le n\frac{2^{nH(k/n)}2^{(n-k-1)H((n-d-2)/(n-k-1))}}{2^{n-k-1}}\\
    &= n2^{nH(k/n)}2^{(n-k-1)\big(H((n-d-2)/(n-k-1))-1\big)}.
\end{align*}
Since $n \sim \alpha d$ and  $k \sim \beta d$ we have that   $n-d-2 \sim (\alpha-1)d$ and $n-k-1\sim (\alpha -\beta)d $. So for $d$ sufficiently large, the above is at most
\[
2\alpha d 2^{\alpha dH(\beta/\alpha)+(\alpha-\beta)d\big( H(\frac{\alpha-1}{\alpha-\beta})-1\big)}.
\]
This quantity goes to zero as $d\to \infty $ as long as $\alpha H(\beta/\alpha)+(\alpha-\beta)(H(\frac{\alpha-1}{\alpha-\beta})-1)<0$ which is the assumption on $\beta$. 
\end{proof}

\begin{proof}[Proof of \cref{thm:density}]
Let $S_n:=\{X_1,\dotsc, X_n\}$. The expected $\ell$-face density of $\conv S_n$ is equal to the probability that $\conv(X_1,\dotsc,X_{\ell+1})$ is a face of $\conv S_n$. So we just need to show that the probability that $\conv(X_1,\dotsc,X_{\ell+1})$ is not a face of $\conv S_n$ is $o(1)$ as $d \to \infty$. By \cref{thm:face}, the probability that $\conv(X_1,\dotsc,X_{\ell+1})$ is not a face of $\conv S_n$ is at most
\[
    \frac{\sum_{i=0}^{n-d-2}\binom{n-\ell-2}{i}}{2^{n-\ell-2}}
\]
Since $n\sim\alpha d$ and $\ell \sim \beta d$, we know that 
\[
\frac{n-d-2}{n-\ell-2} \sim \frac{\alpha-1}{\alpha-\beta}.
\]
By the assumption that $\beta<2-\alpha$, we have that $ \frac{\alpha-1}{\alpha-\beta} <1/2$. Therefore, there exists $\epsilon>0$ such that for $d$ sufficiently large,  $\frac{n-d-2}{n-\ell-2}<1/2-\epsilon$. By the unimodality of the binomial coefficients, the fact that $\frac{n-d-2}{n-\ell-2}<1/2-\epsilon$ implies that 
\[
    \frac{\sum_{i=0}^{n-d-2}\binom{n-\ell-2}{i}}{2^{n-\ell-2}} \le\frac{ n \binom{n-\ell-2}{n-d-2}}{2^{n-\ell-2}}
\]
Again using that  $\binom{n}{k} \le 2^{nH(k/n)}$, we have that for $d$ sufficiently large,
\begin{align*}
\frac{\sum_{i=0}^{n-d-2}\binom{n-\ell-2}{i}}{2^{n-\ell-2}}
&\le \frac{ n \binom{n-\ell-2}{n-d-2}}{2^{n-\ell-2}}\\
&\le \frac{n2^{(n-\ell-2)H(\frac{n-d-2}{n-\ell-2})}}{2^{n-\ell-2}}\\
&\le  \frac{2\alpha d2^{d(\alpha-\beta)H(1/2-\epsilon)}}{2^{d(\alpha-\beta)}}.
\end{align*}
And the above quantity goes to zero as $d \to \infty $ because $H(r)<1$ as long as $r \neq 1/2$. 
\end{proof}

\section{The least neighborly distributions}\label{sec:least}
As previously mentioned, we will construct a family of distributions $\{\mu_d\}_{d\in\mathbb{N}}$ which show that \cref{thm:neighborly,thm:density} are in some sense best possible. Before giving the construction, we need to establish the following proposition, which gives the reverse of the bound given by \cref{thm:WW} in \cref{sec:WW}. 
\subsection{A lower bound on the probability that a point is in the convex hull}\label{sec:WW'}

\begin{definition}\label{def:depth}
    Let $\mu$ be a probability distribution on $\Rl^d$ and $p$ a point in $\Rl^d$. The \emph{depth of $p$ in $\mu$} is defined to be
\[
\min\{\mu(H^+) \suchthat H^+ \text{ is a closed halfspace containing } p \}.
\]
\end{definition}

\begin{proposition}\label{prop:lower}
Let $\mu$ be an absolutely continuous probability distribution on $\Rl^d$. Let $S$ be a set of $n$ independent and identically distributed points from $\mu$. Let $p \in \Rl^d$ be a point and assume that the depth of $p$ in $\mu$ is greater than or equal to $a$. Then the probability that $\conv S$ contains $p$ is greater than or equal to
\begin{align*}
&(d+1)\binom{n}{d+1}\int_0^{a}(y^{n-d-1}+(1-y)^{n-d-1})y^d\ud y\\
&= \sum_{i=0}^{n-d-1}\binom{n}{i}a^i(1-a)^{n-i}+a^n\binom{n-1}{d}.
\end{align*}
\end{proposition}
\begin{proof}
Note that it suffices to prove the statement when $p$ is the origin because otherwise we could translate $p$ and $\mu$ to reduce to this case. Let $p_{n,\mu}$ denote the probability that $\conv S$ contains the origin, where $S$ is a set of $n$ independent and identically distributed points from $\mu$. It is shown in \cite{WW} that there exists a function $h(y)$ (which depends only on $\mu$) such that 
\[
p_{n,\mu}= 2\binom{n}{d+1}\int_0^1y^{n-d-1}h(y) \ud y.
\]
For completeness, we give the definition of $h(y)$ from \cite{WW}: As in \cite[Section 2.2]{WW}, we choose some  absolutely continuous probability distribution $\tilde{\mu}$ on $\Rl^{d+1}$ such that the orthogonal projection of $\tilde{\mu}$ to the first $d$ coordinates is the distribution $\mu$ and we let $\tilde{\ell}$ be the $x_{d+1}$ axis in $\Rl^{d+1}$. Then as in \cite[Section 1.2]{WW}, we let $\sigma$ denote a \emph{$\tilde{\mu}$-random oriented simplex}, i.e. a $d$-simplex whose $d+1$ vertices are i.i.d. points from $\tilde{\mu}$ and one side of $\sigma $ is chosen as the \say{positive} side which is denoted $H^+(\sigma)$. Furthermore, we say that a directed line $\tilde{\ell}$ \emph{enters} an oriented simplex $\sigma$ if it intersects the relative interior of $\sigma$ and is directed from the positive to the negative side of $\sigma$. With this we define 
\[
H_{\tilde{\mu},\tilde{\ell}}:= \prob(\tilde{\ell} \text{ enters } \sigma \text{ and } \tilde{\mu}(H^+(\sigma))\le y).
\]
We then define 
\[
h(y):= h_{\tilde{\mu},\tilde{\ell}}(y):= \frac{\ud H_{\tilde{\mu},\tilde{\ell}}}{\ud y}.
\]
This completes the definition of $h(y)$, see \cite{WW} for more details. 

Now by \cite[Theorem 2.6]{WW}, $h(y) = h(1-y)$ and so 
\[
2\binom{n}{d+1}\int_0^1y^{n-d-1}h \ud y = 2\binom{n}{d+1}\int_0^{1/2}(y^{n-d-1}+(1-y)^{n-d-1})h \ud y.
\]
By \cite[Theorem 3.6]{WW} and the remarks following the proof of that theorem, because of the assumption that the depth of $o$ in $\mu$ is greater than or equal to $a$, the function $h$ satisfies $h = \frac{(d+1)}{2}\min(y,1-y)^{d}$ for $y \le a$ and $y \ge 1-a$. Alternatively, see \cite[Lemma 4.33]{Wagner} for a rigorously stated proof of this claim.
Therefore,
\begin{align*}
p_{n,\mu}&\ge (d+1)\binom{n}{d+1}\int_0^{a}(y^{n-d-1}+(1-y)^{n-d-1})y^d\ud y.
\end{align*}
We use \cite[6.6.4]{MR757537} to get
\[
\int_0^{a}(1-y)^{n-d-1}y^d\ud y  =\frac{(n-d-1)!d!}{n!} \sum_{i=0}^{n-d-1}\binom{n}{i}a^i(1-a)^{n-i}. 
\]
And since $\int_0^{a}y^{n-1}\ud y  =\frac{a^n}{n}$ and $\frac{a^n}{n} (d+1)\binom{n}{d+1}  = a^n \binom{n-1}{d}$,
\[
p_{n,\mu} \ge \sum_{i=0}^{n-d-1}\binom{n}{i}a^i(1-a)^{n-i}+a^n\binom{n-1}{d}.
\]
\end{proof}
We remark that a calculation similar to the one in the proof of \cref{prop:lower} was done in \cite[Theorem 4.32]{Wagner}. However, there is a mistake in that proof; the summation formula they get for the integral is incorrect. 

\subsection{A family of distributions essentially attaining the lower bound}\label{sec:bad}
Here is the definition of the family of distributions: Let $\epsilon_d:= \frac{1}{\sqrt{d}}$ (This choice is somewhat arbitrary). Let 
\[
f(x_1,\dotsc, x_d) := \frac{1}{(2\pi)^{d/2}}e^{-(1/2)(x_1^2+ \cdots+ x_d^2)}
\]
be the probability density function of the standard (mean zero and identity covariance matrix) Gaussian distribution on $\Rl^d$. Let $\mu_d$ be the distribution on $\Rl^d$ with density
\[
\big(1-\epsilon_d\big)f(x)+ \epsilon_d\frac{1_{\{\|x\|_2 \le \epsilon_d\}}}{V_d}
\]
where $V_d$ is the volume of the $d$-ball with radius $\epsilon_d$. In other words, $\mu_d$ is the combination of a Gaussian distribution having $1-\epsilon_d$ of the mass (we call it the \say{Gaussian part}) and the uniform distribution on the $d$-ball of radius $\epsilon_d$ (called the \say{ball part}) having the remaining mass. Note that each distribution $\mu_d$ is absolutely continuous. 
The following proposition shows that \cref{thm:density} is in some sense best possible. 
\begin{proposition}\label{prop:best1}
Let $\alpha>1$ and assume that $\beta>0$ satisfies $\beta > 2-\alpha$. Let $\{\mu_d\}_{d\in \mathbb{N}}$ be the family of probability distributions defined at the start of this section. If for each $d \in \mathbb{N}$ we let $n:=n(d) = \lfloor\alpha d\rfloor$ and let $S_n=\{X_1,\dotsc,X_n \}$ be a set of $n$ iid random points from $\mu_d$, then for any sequence $\ell:=\ell(d)$ with $\ell\sim \beta d$, the expected $\ell$-face density of $\conv S_n$ is $o(1)$ as $d \to \infty$. 
\end{proposition}
\begin{proof}

By definition, the expected $\ell$-face density of $\conv S_n$ is equal to the probability that $\conv(X_1, \dotsc,X_{\ell+1})$ is a face of $\conv S_n$. Therefore, in order to show that the expected $\ell$-face density is $o(1)$, it will suffice to show that the probability that $\conv(X_1, \dotsc,X_{\ell+1})$ is not a face of $\conv S_n$ is $1-o(1)$.

First we will show that we can assume that at least one of the points $\{X_i\}_{i \in [\ell+1]}$ is sampled from the ball part of the distribution. Let $B$ be the event that at least one of the points $\{X_i\}_{i \in [\ell+1]}$ is sampled from the ball part of the distribution and $1_B$ the indicator function of the event $B$. Using that $\big(1+x/y \big)^y <e^x$, we have that
\[
\prob(\neg B) = \big(1-\epsilon_d\big)^{\ell+1} \le e^{-\epsilon_d(\ell+1)} =o(1).
\]
and so $\prob(B) = 1-o(1)$. This means that we ignore the case when event $B$ is not satisfied. In particular, we have
\begin{align*}
      \prob(\conv(X_1, \dotsc,X_{\ell+1}) \text{ is not a face})
      & \ge   \prob(\conv(X_1, \dotsc,X_{\ell+1}) \text{ is not a face}\cap B)
\end{align*}

and so it suffices to show that $\prob(\conv(X_1, \dotsc,X_{\ell+1}) \text{ is not a face} \cap B)=1-o(1)$. 

Define
\[ f(x_1,\dotsc, x_{\ell+1}):= \begin{cases} 
      1 & \text{ if } \conv(x_1,\dotsc,x_{\ell+1}) \text{ is not a face} \\
     0 &  \text{ otherwise}
      \end{cases}
\]
Then we have that
\begin{align*}
   \prob(\conv(X_1, \dotsc,X_{\ell+1}) \text{ is not a face}\cap B)&=  \int_{\Rl^d} \dotsc \int_{\Rl^d}1_B f(x_1,\dotsc, x_{\ell+1})\mu_d(\ud x_n)\dotsc \mu_d(\ud x_{1}).
\end{align*}
Letting $B' \subset \Rl^{d(\ell+1)}$ be the set of point sets satisfying $1_B=1$, we can rewrite the above as
\begin{align}\label{eq:int2}
   \prob(\conv(X_1, \dotsc,X_{\ell+1}) \text{ is not a face}\cap B) &= \int_{B'} \int_{\Rl^d} \dotsc \int_{\Rl^d} f(x_1,\dotsc, x_{\ell+1})\mu_d(\ud x_n)\dotsc \mu_d(\ud x_{1}).
\end{align}
For any fixed choice of points $x_{1}, \dotsc, x_{\ell+1}$, the inner integral
\begin{align}\label{eq:inner}
\int_{\Rl^d} \dotsc\int_{\Rl^d} f(x_1,\dotsc, x_{\ell+1})\mu_d(\ud x_n)\dotsc \mu_d(\ud x_{\ell+2})
\end{align}
is equal to the probability that $\conv(\{X_i\}_{i\in [\ell+2,n]}) \cap L\neq \emptyset$ where $L := \aff(x_1,\dotsc,x_{\ell+1})$. Under the assumption that the points $x_{1}, \dotsc, x_{\ell+1}$ satisfy $1_B = 1$, we will show that this probability is $1-o(1)$ and therefore that \cref{eq:inner} is $1-o(1)$ for any choice of $x_1,\dotsc x_{\ell+1}$ such that $1_B=1$. Let $\pi_L$ be the orthogonal projection of $\Rl^d$ to the subspace $L^\perp$ of dimension $d-\ell$ that is orthogonal to $L$. Let $\pi \mu_d$ denote the measure on $L^\perp$ which is the projection of $\mu_d$. The probability that $\conv(\{X_i\}_{i\in [\ell+2,n]}) \cap L\neq \emptyset$ is equal to the probability that $\conv(\pi_L X_{\ell+2}, \dotsc,\pi_L X_{n}) $ contains $\pi L$. 
Note that since at least one of the $x_i$, $1 \le i \le \ell+1$ is sampled from the ball part of the distribution, the distance from $L$ to the origin is at most $\epsilon_d$. We claim that this means that the depth of $\pi L$ in $\pi \mu_d$ is at least $(1-\epsilon_d)(1/2-\epsilon_d)$. In order to show this, we need to show that every hyperplane in $\pi_L\Rl^d$ through $\pi L$ has at least $(1-\epsilon_d)(1/2-\epsilon_d)$ of the mass of $\pi \mu_d$ on each side. We will actually prove the stronger statement that every hyperplane in $\Rl^d$ which contains $L$ has has at least $(1-\epsilon_d)(1/2-\epsilon_d)$ of the mass of $\mu_d$ on each side. The Gaussian measure of halfspace determined by a hyperplane at distance less than $\epsilon_d$ from the origin is greater than 
\[
1/2- \frac{1}{\sqrt{2\pi}} \int_{0}^{\epsilon_d} e^{-x^2/2 }\ud x \ge 1/2-\epsilon_d.
\]
And $1-\epsilon_d$ of the mass of $\mu_d$ is the standard Gaussian measure. So the claim that the depth of $\pi L$ in $\pi \mu_d$ is at least $(1-\epsilon_d)(1/2-\epsilon_d)\ge 1/2-2\epsilon_d$ follows.

Now by \cref{prop:lower}, the probability that $\conv(\pi_L X_{\ell+2}, \dotsc,\pi_L X_{n}) $ contains $\pi L$ is at least 
\[
(d-\ell+1)\binom{n-\ell-1}{d-\ell+1}\int_0^{1/2-2\epsilon_d}(y^{n-d-2}+(1-y)^{n-d-2})y^{d-\ell}\ud y.
\]
Using the formula in \cref{prop:lower}  along with the fact that $ \binom{n}{i}= \binom{n-1}{i-1} + \binom{n-1}{i}$ one can show that if the range of integration is extended from  $(0,1/2-\epsilon_d)$ to $(0,1/2)$, then 
\[
(d-\ell+1)\binom{n-\ell-1}{d-\ell+1}\int_0^{1/2}(y^{n-d-2}+(1-y)^{n-d-2})y^{d-\ell}\ud y=\frac{\sum_{i=0}^{n-d-2}\binom{n-\ell-2}{i}}{2^{n-\ell-2}}.
\]
Therefore, using that $\big(1+x/y \big)^y <e^x$, we have that 
\begin{align*}
    &(d-\ell+1)\binom{n-\ell-1}{d-\ell+1}\int_0^{1/2-2\epsilon_d}(y^{n-d-2}+(1-y)^{n-d-2})y^{d-\ell}\ud y\\ 
    &= \frac{\sum_{i=0}^{n-d-2}\binom{n-\ell-2}{i}}{2^{n-\ell-2}}  -(d-\ell+1)\binom{n-\ell-1}{d-\ell+1}\int_{1/2-2\epsilon_d}^{1/2}(y^{n-d-2}+(1-y)^{n-d-2})y^{d-\ell}\ud y \\ 
    &\ge  \frac{\sum_{i=0}^{n-d-2}\binom{n-\ell-2}{i}}{2^{n-\ell-2}} 
    - 2\epsilon_d(d-\ell+1)\binom{n-\ell-1}{d-\ell+1} \bigg(\frac{1}{2}+2\epsilon_d \bigg)^{n-\ell-2}\\
    &= \frac{\sum_{i=0}^{n-d-2}\binom{n-\ell-2}{i}}{2^{n-\ell-2}} 
    - 2\epsilon_d(d-\ell+1)\binom{n-\ell-1}{d-\ell+1} \frac{1}{2^{n-\ell-2}}\big(1+4\epsilon_d \big)^{n-\ell-2}\\
    &\ge  \frac{\sum_{i=0}^{n-d-2}\binom{n-\ell-2}{i}}{2^{n-\ell-2}} -2\epsilon_d(d-\ell+1)\binom{n-\ell-1}{d-\ell+1} \frac{1}{2^{n-\ell-2}}e^{4\epsilon_d(n-\ell-2)}.
\end{align*}
First we claim that the second term above is $o(1)$. To show this, note that since $\ell\sim \beta d$ and $n\sim \alpha d$, $\frac{d-\ell-1}{n-\ell-2} \sim \frac{1-\beta}{\alpha-\beta}$. And by the assumption that $\beta>2-\alpha$, we have that $\frac{1-\beta}{\alpha-\beta}<\frac{1-\beta}{2-2\beta}=1/2$ and so there exists $\epsilon>0$ such that for $d$ sufficiently large, $\frac{d-\ell-1}{n-\ell-2}<1/2-\epsilon$. Now using that $\binom{n}{k} \le 2^{nH(k/n)}$, we have that for $d$ sufficiently large,
\begin{align*}
 2\epsilon_d(d-\ell+1)\binom{n-\ell-1}{d-\ell+1} \frac{1}{2^{n-\ell-2}}e^{4\epsilon_d(n-\ell-2)} 
 &\le 2\epsilon_d(d-\ell+1)\frac{2^{(n-\ell-1)H(1/2-\epsilon)}  }{2^{n-\ell-2}}e^{8\sqrt{d}} \\
 &=o(1)
\end{align*}
because $H(1/2-\epsilon)<1$. This means that it suffices to show that $2^{\ell+2-n} \sum_{i=0}^{n-d-2}\binom{n-\ell-2}{i} = 1-o(1)$, or equivalently, that $2^{\ell+2-n} \sum_{i=n-d-1}^{n-\ell-2}\binom{n-\ell-2}{i} = 2^{\ell+2-n} \sum_{j=0}^{d-\ell-1}\binom{n-\ell-2}{j} = o(1)$. So by unimodality of the binomial coefficients, and again using that  $\binom{n}{k} \le 2^{nH(k/n)}$,
\begin{align*}
    2^{\ell+2-n} \sum_{j=0}^{d-\ell-1}\binom{n-\ell-2}{j} & \le  2^{\ell+2-n}(d-\ell)\binom{n-\ell-2}{d-\ell-1}\\
    &\le 2^{\ell+2-n}(d-\ell) 2^{(n-\ell-2)H(\frac{d-\ell-1}{n-\ell-2})}\\
    &\le 2^{\ell+2-n}(d-\ell) 2^{(n-\ell-2)H(1/2-\epsilon)}. 
\end{align*}
And the above quantity is $o(1)$ because $H(1/2-\epsilon)<1$. 

This shows that \cref{eq:int2} is the integral over $B'$ of a function which is uniformly bounded from below by a function which is equal to $1-o(1)$. Since the measure of $B'$ is equal to $\prob(B) = 1-o(1)$, this shows that \cref{eq:int2} is equal to $1-o(1)$ as desired.
\end{proof}

We can also prove a similar result for $k$-neighborliness of the distributions $\mu_d$. (The distributions $\mu_d\}_{d\in \mathbb{N}}$ in the following proposition are the distributions defined at the beginning of this section.)
\begin{proposition}
Let $\alpha>1$ and assume that $\beta>0$ satisfies $\alpha H(\beta/\alpha)+(\alpha-\beta)(H(\frac{\alpha-1}{\alpha-\beta})-1)>0$, or equivalently, $\frac{\alpha^\alpha}{(\alpha-1)^{\alpha-1}2^\alpha} > \frac{\beta^\beta(1-\beta)^{1-\beta}}{2^\beta} $. Let $\{\mu_d\}_{d\in \mathbb{N}}$ be the family of probability distributions defined at the start of this section. If for each $d$ we let $n:=n(d) = \lfloor\alpha d\rfloor$ and let $S_n=\{X_1,\dotsc,X_n \}$ be a set of $n$ iid random points from $\mu_d$, then for any sequence $k:=k(d)$ with $k\sim \beta d$, the expected number of subsets of $S_n$ of size $k$ which are not faces of $\conv S_n$ goes to infinity as $d \to \infty$. 
\end{proposition}

\begin{proof}

Let $\{\mu_d\}_{d\in \mathbb{N}}$ be the distributions defined at the start of this section.

Let $B$ be the event that at least one of the points $\{X_i\}_{i \in [k]}$ is sampled from the ball part of the distribution and $1_B$ the indicator function of the event $B$. The same argument as in  the proof of \cref{prop:best1} shows that $\prob(B) = 1-o(1)$. We want to show that the expected number of subsets of size $k$ which are not faces goes to infinity. It will suffice to show that the expected number of subsets of size $k$ which contain at least one point sampled from the ball part and which are not faces goes to infinity. That is, it suffices to show that 
\[
\binom{n}{k}\cdot\prob(\conv(X_1, \dotsc,X_{k}) \text{ is not a face} \cap B) \to \infty.
\]
Define
\[ f(x_1,\dotsc, x_{\ell+1}):= \begin{cases} 
      1 & \text{ if } \conv(x_1,\dotsc,x_{k}) \text{ is not a face} \\
     0 &  \text{ otherwise}
      \end{cases}
\]
Then we have that
\begin{align*}
  &\binom{n}{k} \prob(\conv(X_1, \dotsc,X_{\ell+1}) \text{ is not a face}\cap B)\\ 
  &= \binom{n}{k} \int_{\Rl^d} \dotsc \int_{\Rl^d}1_B f(x_1,\dotsc, x_{\ell+1})\mu_d(\ud x_n)\dotsc \mu_d(\ud x_{1}).
\end{align*}
Letting $B' \subset \Rl^{d(\ell+1)}$ be the set of point sets satisfying $1_B=1$, we can rewrite the above as
\begin{align}\label{eq:int3}
  & \binom{n}{k}\prob(\conv(X_1, \dotsc,X_{k}) \text{ is not a face}\cap B)\nonumber\\ 
  &= \binom{n}{k}\int_{B'} \int_{\Rl^d} \dotsc \int_{\Rl^d} f(x_1,\dotsc, x_{k})\mu_d(\ud x_n)\dotsc \mu_d(\ud x_{1}).
\end{align}
For any fixed choice of points $x_{1}, \dotsc, x_{k}$, the inner integral
\begin{align}
\int_{\Rl^d} \dotsc\int_{\Rl^d} f(x_1,\dotsc, x_{k})\mu_d(\ud x_n)\dotsc \mu_d(\ud x_{k+1})
\end{align}
is equal to the probability that $\conv(\{X_i\}_{i\in [k+1,n]}) \cap L\neq \emptyset$ where $L := \aff(x_1,\dotsc,x_{k})$.

Under the assumption that the points  $x_{1}, \dotsc, x_{k}$ satisfy $1_B = 1$, we will show that $\binom{n}{k}$ times this probability approaches infinity. Let $\pi_L$ be the orthogonal projection of $\Rl^d$ to the subspace $L^\perp$ of dimension $d-k+1$ that is orthogonal to $L$. Let $\pi \mu_d$ denote the measure on $L^\perp$ which is the projection of $\mu_d$. The probability that $\conv(\{X_i\}_{i\in [k+1,n]}) \cap L\neq \emptyset$ is equal to the probability that $\conv(\pi_L X_{k+1}, \dotsc,\pi_L X_{n}) $ contains $\pi L$. 
Note that since at least one of the $x_i$, $1 \le i \le k$ is sampled from the ball part of the distribution, the distance from $L$ to the origin is at most $\epsilon_d$. We claim that this means that the depth of $\pi L$ in $\pi \mu_d$ is at least $(1-\epsilon_d)(1/2-\epsilon_d)$. In order to show this, we need to show that every hyperplane in $\pi_L\Rl^d$ through $\pi L$ has at least $(1-\epsilon_d)(1/2-\epsilon_d)$ of the mass of $\pi \mu_d$ on each side. We will actually prove the stronger statement that every hyperplane in $\Rl^d$ which contains $L$ has has at least $(1-\epsilon_d)(1/2-\epsilon_d)$ of the mass of $\mu_d$ on each side. The Gaussian measure of halfspace determined by a hyperplane at distance less than $\epsilon_d$ from the origin is greater than 
\[
1/2- \frac{1}{\sqrt{2\pi}} \int_{0}^{\epsilon_d} e^{-x^2/2 }\ud x \ge 1/2-\epsilon_d.
\]
And $1-\epsilon_d$ of the mass of $\mu_d$ is the standard Gaussian measure. So the claim that the depth of $\pi L$ in $\pi \mu_d$ is at least $(1-\epsilon_d)(1/2-\epsilon_d)\ge 1/2-2\epsilon_d$ follows.

Now by \cref{prop:lower}, the probability that $\conv(\pi_L X_{k+1}, \dotsc,\pi_L X_{n}) $ contains $\pi L$ is at least 
\[
(d-k+2)\binom{n-k}{d-k+2}\int_0^{1/2-2\epsilon_d}(y^{n-d-2}+(1-y)^{n-d-2})y^{d-k+1}\ud y.
\]
Now using the fact that $\big(1+x/y \big)^y > e^{xy/(x+y)}$, and $\binom{n}{k} = \binom{n}{n-k}$,
\begin{align*}
&\binom{n}{k}(d-k+2)\binom{n-k}{d-k+2}\int_0^{1/2-2\epsilon_d}(y^{n-d-2}+(1-y)^{n-d-2})y^{d-k+1}\ud y\\
& \ge \binom{n}{k}(d-k+2)\binom{n-k}{d-k+2}\int_{1/2-3\epsilon_d}^{1/2-2\epsilon_d}(y^{n-d-2}+(1-y)^{n-d-2})y^{d-k+1}\ud y\\
&\ge \binom{n}{k}(d-k+2)\binom{n-k}{d-k+2} \epsilon_d \Bigg(\frac{1}{2} - 3\epsilon_d  \Bigg)^{n-k-1}\\
& = \binom{n}{k}(d-k+2) \epsilon_d \binom{n-k}{n-d-2}\frac{1}{2^{n-k-1}} \bigg(1- 6\epsilon_d  \bigg)^{n-k-1}\\
& \ge\binom{n}{k} (d-k+2) \epsilon_d \binom{n-k}{n-d-2}\frac{1}{2^{n-k-1}} e^{-6\epsilon_d(n-k-1)/(6\epsilon_d+1)}.
\end{align*}
Now since $n \sim \alpha d$ and $k\sim \beta d $, we have that $\frac{n-d-2}{n-k} \sim \frac{\alpha-1}{\alpha -\beta}$. And using the fact that $\binom{n}{k} \ge (1/(n+1))2^{nH(k/n)}$, for $d$ sufficiently large the above is lower bounded by
\begin{align*}
  & \frac{1}{2}\cdot \frac{2^{\alpha d H(\beta/\alpha)} 2^{d(\alpha-\beta)H(\frac{\alpha-1}{\alpha-\beta})}}{2^{(\alpha-\beta)d}}\cdot \frac{(d-k+2) \epsilon_d e^{-6\epsilon_d(n-k-1)/(6\epsilon_d+1)}}{(n+1)(n-k+1)}\\  
  &= \frac{1}{2}\cdot 2^{d\big(\alpha H(\beta/\alpha)+(\alpha-\beta)(H(\frac{\alpha-1}{\alpha-\beta})-1)   \big)} \cdot \frac{(d-k+2) \epsilon_d e^{-6\epsilon_d(n-k-1)/(6\epsilon_d+1)}}{(n+1)(n-k+1)}. 
\end{align*}
And the above quantity goes to infinity as $d\to\infty$ because $\alpha H(\beta/\alpha)+(\alpha-\beta)(H(\frac{\alpha-1}{\alpha-\beta})-1)>0$ and the term $e^{-6\epsilon_d(n-k-1)/(6\epsilon_d+1)} $ is $\Omega(1/2^{\delta d})$ for all $\delta>0$.

We have shown that \cref{eq:int3} is the integral over a set of measure $1-o(1)$ of a function which is uniformly bounded from below by a function which approaches infinity as $d \to \infty$. This shows that the quantity in \cref{eq:int3} approaches infinity. 
\end{proof}

\paragraph{Acknowledgments.}
Thanks to Luis Rademacher for many helpful discussions and comments. 
This material is based upon work supported by the National Science Foundation under Grants CCF-1657939, CCF-1934568 and CCF-2006994.

\bibliographystyle{abbrv}
\bibliography{bib}

\end{document}